\documentclass[a4paper,11pt]{article}
\usepackage{latexsym}
\usepackage{amssymb}
\usepackage{enumerate}
\usepackage{amsfonts}
\usepackage{amsmath}
\usepackage[dvips]{graphicx}
\usepackage{epsf}
\newcommand{\qed}{$\Box$}

\newenvironment{@abssec}[1]{%
    \if@twocolumn

      \section*{#1}%
    \else

      \vspace{.05in}\footnotesize
      \parindent .2in
 {\upshape\bfseries #1. }\ignorespaces
    \fi}

    {\if@twocolumn\else\par\vspace{.1in}\fi}

\newenvironment{keywords}{\begin{@abssec}{\keywordsname}}{\end{@abssec}}

\newenvironment{AMS}{\begin{@abssec}{\AMSname}}{\end{@abssec}}

\newcommand\keywordsname{Key words}
\newcommand\AMSname{AMS subject classifications}
\newcommand\AMname{AMS subject classification}
\newtheorem{theorem}{Theorem}
 \newtheorem{lemma}[theorem]{Lemma}

\def\qed{\vbox{\hrule height0.6pt\hbox{%
  \vrule height1.3ex width0.6pt\hskip0.8ex
  \vrule width0.6pt}\hrule height0.6pt
 }}
\def\theequation{\arabic{section}.\arabic{equation}}
 \def\thetheorem{\arabic{section}.\arabic{theorem}}

\def\theequation{\arabic{section}.\arabic{equation}}
 \def\thetheorem{\arabic{section}.\arabic{theorem}}

\def\pa{\partial}

\def\Om{\Omega}

\def\dist{\mbox{\rm dist}}

\title{Matzoh ball soup revisited: \\ the boundary regularity issue
\thanks{This research was partially supported by a Grant-in-Aid
for Scientific Research (B) ($\sharp$ 20340031) of
Japan Society for the Promotion of Science and by a
Grant of the Ital\-ian MURST.}}

\author{Rolando Magnanini\thanks{Dipartimento di Matematica U.~Dini,
Universit\` a di Firenze, viale Morgagni 67/A, 50134 Firenze, Italy.
({\tt magnanin@math.unifi.it}).} 
 and Shigeru Sakaguchi\thanks{Department of Applied Mathematics,
Graduate School of  Engineering, Hiroshima
University, Higashi-Hiroshima, 739-8527,  Japan.
({\tt sakaguch@amath.hiroshima-u.ac.jp}).}}

\begin{document}

\maketitle

\begin{abstract}
We consider  nonlinear diffusion equations of the form  $\partial_t u= \Delta \phi(u)$ in $\mathbb R^N$ with $N \ge 2.$ When $\phi(s) \equiv s$, this is just the heat equation. Let  $\Omega$ be a domain in $\mathbb R^N$, where $\partial\Omega$ is bounded and $\partial\Omega = \partial\left(\mathbb R^N\setminus \overline{\Omega}\right)$.  We consider the initial-boundary value problem, where the initial value equals zero and the boundary value equals $1$,  and
the Cauchy problem where the initial data is  
the characteristic function of the set $\Omega^c = \mathbb R^N\setminus \Omega$. 
We settle the boundary regularity issue for the characterization of the sphere as a stationary level surface of the solution $u:$ no regularity assumption is needed for $\partial\Omega.$
\end{abstract}




\pagestyle{plain}
\thispagestyle{plain}
\markboth{R. MAGNANINI AND S. SAKAGUCHI}{Matzoh ball soup revisited}

\pagestyle{plain}
\thispagestyle{plain}

\section{Introduction}
Let $\Omega$ be a domain in $\mathbb R^N\ (N \ge 2)$.  Let $\phi : \mathbb R \to \mathbb R$ satisfy
\begin{equation}
\label{nonlinearity}
\phi \in C^2(\mathbb R), \quad \phi(0) = 0, \ \mbox{ and }\ 0 < \delta_1 \le \phi^\prime(s) \le \delta_2\ \mbox{ for } s \in \mathbb R,
\end{equation}
 where $\delta_1, \delta_2$ are positive constants. Consider the unique bounded solution $u = u(x,t)$ of either
  the  initial-boundary value problem: 
\begin{eqnarray}
&\partial_t u=\Delta \phi(u)\ \ &\mbox{in }\ \Om\times (0,+\infty),\label{diffusion}\\
&u=1\ \ &\mbox{on }\ \partial\Omega\times (0,+\infty),\label{dirichlet}\\
&u=0\ \ &\mbox{on }\ \Om\times \{0\},\label{initial}
\end{eqnarray}
or the Cauchy  problem:
\begin{equation}
\label{cauchy}
\partial_t u=\Delta \phi(u)\ \mbox{ in }\ \mathbb R^N \times (0, +\infty)\quad\mbox{ and }\ u = \chi_{\Omega^c}\ \mbox{ on }\ \mathbb R^N \times \{0\};
\end{equation}
here $\chi_{\Omega^c}$ denotes the characteristic function of the set $\Omega^c = \mathbb R^N \setminus \Omega$. 
(As a solution $u$ of problem \eqref{diffusion}-\eqref{initial} we mean a classical solution 
belonging to $C^{2,1}(\Omega\times(0,+\infty)) \cap L^\infty(\Omega \times (0,+\infty))\cap C^0(\overline{\Omega} \times (0,+\infty))$ and such that $u(\cdot,t) \to 0$ in $L^1_{loc}(\Omega)$ as $t \to 0;$ 
similarly, a solution of  \eqref{cauchy} is a classical solution belonging to $C^{2,1}(\mathbb R^N \times (0,+\infty))\cap L^\infty(\mathbb R^N \times (0,+\infty))$ and such that $u(\cdot,t) \to \chi_{\Omega^c}(\cdot)$ in $L^1_{loc}(\mathbb R^N)$ as $t 
\to 0$.
Note that the uniqueness of the solution of either problem \eqref{diffusion}-\eqref{initial} or  problem \eqref{cauchy} follows from the 
comparison principle, as shown in \cite[Theorem A.1]{MSjde2nd}.)
 
In Theorem \ref{th:nonlinear-Varadhan} below, for the reader's convenience, we recall
a nonlinear version of an asymptotic formula --- due to Varadhan \cite{Va} for the linear case --- that was proved in
\cite[Theorem 1.1 and Theorem 4.1]{MSpoincare} and \cite[Theorem 2.1 and Remark 2.2]{MSjde2nd}.
To this aim, we
define $\Phi: (0,\infty) \to \mathbb R$ by
\begin{equation}
\label{press}
\Phi(s) = \int_1^s\frac {\phi^\prime(\xi)}{\xi} d\xi\qquad(s >0),
\end{equation}
(note that if $\phi(s) \equiv s$, then $\Phi(s) = \log s$)
and let $d = d(x)$ be the distance function given by
\begin{equation}
\label{distance function to the boundary}
d(x) = \dist(x,\partial\Omega)\ \mbox{ for } x \in \Omega.
\end{equation}


\renewcommand{\thetheorem}{A}

\begin{theorem} {\rm ({\rm \cite{MSpoincare, MSjde2nd}})}
\label{th:nonlinear-Varadhan}
Let $u$ be the solution of 
either problem \eqref{diffusion}-\eqref{initial} or  problem \eqref{cauchy}. Under the assumption that  $\partial\Omega = \partial\left(\mathbb R^N\setminus \overline{\Omega}\right)$, 
$$
-4t\Phi(u(x,t)) \to d(x)^2 \ \mbox{ as } t \to 0^+ \mbox{ uniformly on every compact set in }\Omega.
$$
\end{theorem}

\renewcommand{\thetheorem}{\arabic{section}.\arabic{theorem}}
 \setcounter{theorem}{0}
The assumption that $\partial\Omega = \partial\left(\mathbb R^N\setminus \overline{\Omega}\right)$ is general. For example,  it holds  for Lipschitz domains.

A conjecture, posed by Klamkin \cite{Kl} and referred to by Zalcman \cite{Z} as the Matzoh ball soup, was settled affirmatively by Alessandrini \cite{A1, A2}. In \cite{A2}, when $\phi(s) \equiv s$ and $\Omega$ is bounded, under the assumption that every point of $\partial\Omega$ is regular with respect to the Laplacian, it was proved that if all the spatial level surfaces of the solution $u$ of problem \eqref{diffusion}-\eqref{initial} are invariant with time then $\Omega$ must be a ball.  The proof requires assuming that {\it infinitely many} level surfaces of $u$ are invariant with time. Here, we remark that the values of $u$ vary with time on its spatial level surfaces.

In \cite{MSannals, MSindiana, MSpoincare}, 
we proved symmetry results for
solutions of either problem \eqref{diffusion}-\eqref{initial} or problem \eqref{cauchy} 
which admit a time-invariant level surface. Those results were obtained 
under classical regularity assumptions on the domains at stake. 
In the present paper, with the aid of Theorem \ref{th:nonlinear-Varadhan}
and Theorem \ref{th:interaction curvatures} below,
we show that such results also hold under very general assumptions.
\par
The following theorem removes the hypotheses made in \cite[Theorem 1.2 and Theorem 1.3]{MSpoincare}
that $\partial\Omega$ and $\partial D$ be $C^2$-smooth.


\begin{theorem}
\label{th:nonlinear Matzoh revisited} Let $u$ be the unique bounded solution of 
either problem \eqref{diffusion}-\eqref{initial} or problem \eqref{cauchy}. 
Suppose that $\partial\Omega$ is bounded and $\partial\Omega = \pa\left(\mathbb R^N\setminus \overline{\Omega}\right)$. 
\par
Let $D$ be a $C^1$ domain with bounded boundary $\partial D$ satisfying  $\overline D \subset \Omega$. 
Then the following statements hold.
\begin{enumerate}[\rm (I)]
\item 
If there exists a function $a: (0,+\infty) \to (0,+\infty)$ satisfying
\begin{equation}
\label{stationary level 1}
u(x,t) = a(t)\ \mbox{ for every } (x,t) \in \partial D \times (0, +\infty),
\end{equation}
 then $\partial\Omega$ must be a sphere.
\item
If $D$ is unbounded and for each connected component $\Gamma$ of $\partial D$ there exists a function $a_\Gamma: (0,+\infty) \to (0,+\infty)$ satisfying
\begin{equation}
\label{stationary level 2}
u(x,t) = a_\Gamma(t)\ \mbox{ for every } (x,t) \in \Gamma \times (0, +\infty), 
\end{equation}
then $\partial\Omega$ must be a sphere.
\end{enumerate}
\end{theorem}

The next theorem concerns results obtained in \cite{MSannals, MSindiana, MSpoincare}, and in particular \cite[Theorem 2.1]{MSpoincare}; we prove that they hold for a general domain $\Omega,$ without assuming the exterior sphere condition on $\Omega$.


\begin{theorem} 
\label{th:Matzoh revisited}  Let $\phi(s)\equiv s$ and let $u$ be the unique bounded solution of 
either problem \eqref{diffusion}-\eqref{initial} or problem \eqref{cauchy}.  Suppose that $\partial\Omega$ is bounded 
and $\partial\Omega = \partial\left(\mathbb R^N\setminus \overline{\Omega}\right)$. 
\par
Let $D$ be a domain with bounded boundary $\partial D$ satisfying  $\overline D \subset \Omega$, and let $\Gamma$ be a connected component of  $\partial D$ satisfying 
\begin{equation}
\label{nearest component}
\dist(\Gamma, \partial\Omega) = \mbox{\rm dist}(\partial D, \partial\Omega).
\end{equation}
 Suppose that  $D$ satisfies the interior cone condition on $\Gamma$. 
\par
If there exists a function $a: (0,+\infty) \to (0,+\infty)$ satisfying
\begin{equation}
\label{stationary level 3}
u(x,t) = a(t)\ \mbox{ for every } (x,t) \in \Gamma \times (0, +\infty),
\end{equation}
then $\partial\Omega$ must be either a sphere or the union of two concentric spheres.
\end{theorem}

We sketch the main features of the proof of item (I) of Theorem \ref{th:nonlinear Matzoh revisited}.  By Theorem \ref{th:nonlinear-Varadhan} there exists a number  $R > 0$ such that $d(x) = R\ \mbox{ for all } x \in \partial D$, and hence, since $\partial D$ is of class $C^1$, we can conclude that $\Omega$ is the union
of $D$ and all the open balls $B_R(x)$ of radius $R$ centered at points $x\in\partial D.$ 
Thanks to this remark, we can apply the method of moving planes directly to either $D$ or $\mathbb R^N \setminus \overline{D}$ (in \cite[Theorem 1.2 and Theorem 1.3]{MSpoincare} we applied it to either $\Omega$ or $\mathbb R^N \setminus \overline{\Omega},$ instead); for this reason, we do not need the smoothness of $\partial\Omega$. 
The proof of item (II) of Theorem \ref{th:nonlinear Matzoh revisited} runs similarly.  

Eventually, Theorem \ref{th:nonlinear Matzoh revisited} is proved by the 
method of moving planes, and hence the following problem is open: When $D$ is bounded with disconnected boundary and for each connected component $\Gamma$ of $\partial D$ there exists a function $a_\Gamma : (0,+\infty) \to (0,+\infty)$ satisfying \eqref{stationary level 2}, must $\partial\Omega$ be a sphere?
Of course, it is assumed that $a_\Gamma$'s are different for at least two components.
\par
The removal of the exterior sphere condition on $\Omega$ in Theorem \ref{th:Matzoh revisited} relies on
\cite[Theorem 1.1 and Remark 1.2]{MSjde2nd}, that we summarize in 
Theorem \ref{th:interaction curvatures} for later use. 


\renewcommand{\thetheorem}{B}

\begin{theorem} {\rm (\cite[Theorem 1.1 and Remark 1.2]{MSjde2nd})}
\label{th:interaction curvatures} 
Let $x_0 \in \Omega$ and assume that the open ball $B_R(x_0)$  
is contained in  $\Omega$ and such that $\overline{B_R(x_0)} \cap \partial\Omega = \{ y_0 \}$ for some $y_0 \in \partial\Omega$.
Suppose that $\partial\Omega$ is of class $C^2$ in a neighborhood of the point $y_0$.
\par
Let $u$ be the solution of 
either problem \eqref{diffusion}-\eqref{initial} or problem \eqref{cauchy}. Then we have:
\begin{equation}
\label{asymptotics and curvatures}
\lim_{t\to 0^+}t^{-\frac{N+1}4 }\!\!\!\int\limits_{B_R(x_0)}\! u(x,t)\ dx=
c(\phi,N)\left\{\prod\limits_{j=1}^{N-1}\left(\frac 1R - \kappa_j(y_0)\right)\right\}^{-\frac 12}.
\end{equation}
Here, 
$\kappa_1(y_0),\dots,\kappa_{N-1}(y_0)$ denote the principal curvatures of $\partial\Omega$ at $y_0$ with 
respect to the inward normal direction to $\partial\Omega$  
and $c(\phi,N)$ is a positive constant depending only on $\phi$ and $N$ 
--- of course,  $c(\phi,N)$ depends on the problems  \eqref{diffusion}-\eqref{initial}  or \eqref{cauchy}. 
\par
When $\kappa_j(y_0) = \frac 1R$ for some $j \in \{ 1, \cdots, N-1\}$, 
\eqref{asymptotics and curvatures} holds by setting the right-hand side to $+\infty$ (notice that 
$\kappa_j(y_0) \le 1/R$ always holds
for all $j$'s).
\end{theorem}

\renewcommand{\thetheorem}{\arabic{section}.\arabic{theorem}}
 \setcounter{theorem}{0}

By this theorem and the balance law also used in \cite[Theorem 2.1]{MSpoincare} and \cite{MSannals, MSindiana}, first
we can begin with inferring that $\sum\limits_{j=1}^{N-1}\left( \frac 1R -\kappa_j\right)$ equals a positive constant on some portion of the boundary;  and hence, analyticity of $\Gamma$ helps us extend
such an equality to the whole connected component of $\partial\Omega$ parallel to $\Gamma$.

Sections \ref{section2} and \ref{section3} are devoted to
the proofs of Theorems \ref{th:nonlinear Matzoh revisited} and \ref{th:Matzoh revisited}, respectively. In the Appendix, for the reader's convenience, we give a proof of Theorem \ref{th:interaction curvatures} for the heat equation under the assumption that 
$\displaystyle{ \max_{1\le j \le N-1} \kappa_j(y_0) < \frac 1R}$. Then the proof of Theorem \ref{th:Matzoh revisited} will be self-contained.


\setcounter{equation}{0}
\setcounter{theorem}{0}

\section{Proof of Theorem \ref{th:nonlinear Matzoh revisited}}
\label{section2}

Let us prove item (I) first. We start with a lemma.
\begin{lemma}
\label{le:exterior sum}
 Under the assumptions of \ {\rm (I)} of Theorem \ref{th:nonlinear Matzoh revisited}, there exists a number  $R > 0$ such that
\begin{equation}
\label{distance and parallel}
d(x) = R\ \mbox{ for all } x \in \partial D\  \mbox{ and } \  \Omega = D \cup 
\bigcup\limits_{x\in\partial D} B_R(x) = \{ y \in \mathbb R^N : \mbox{\rm dist}(y, \overline{D}) < R \}.
\end{equation}
\end{lemma}

\noindent
\begin{proof} Theorem \ref{th:nonlinear-Varadhan} and  \eqref{stationary level 1} imply that there exists a number $R > 0$ such that $d(x) = R$ 
for all  $x \in \partial D,$ which in turn immediately gives us the inclusion:
\begin{equation}
 \label{the left includes the right}
 \Omega \supset D \cup \bigcup\limits_{x\in\partial D} B_R(x).
 \end{equation}
 We observe that
\begin{equation}
\label{distance and balls}
 D \cup \bigcup\limits_{x\in\partial D} B_R(x) = \{ y \in \mathbb R^N : \mbox{\rm dist}(y, \overline{D}) < R \}.
 \end{equation}

 \par
Let us show the converse inclusion of \eqref{the left includes the right}.
Since $\partial D$ is of class $C^1$, from the first part of \eqref{distance and parallel} we have
\begin{equation}
\label{use C2 smoothness}
\overline{B_R(x)} \cap \partial\Omega = \{ y(x) \} \ \mbox{ and } B_R(y(x)) \cap D =\varnothing\ \mbox{ for every } x \in \partial D,
\end{equation}
where $y(x) = x + R\nu_{\partial D}(x)$  and $\nu_{\partial D}(x)$ denotes the unit outward normal vector  to $\partial D$ at $x \in \partial D$.
Then it follows that
\begin{equation}
\label{distance to D}
\{ y \in \mathbb R^N : \mbox{ dist}(y, \overline{D}) = R \} = \{ y(x)  : x \in \partial D \} \subset \partial\Omega.
\end{equation}
By \eqref{the left includes the right} and \eqref{distance and balls}, 
\begin{equation}
\label{distance to D and interior}
\{ y \in \mathbb R^N : \mbox{ dist}(y, \overline{D}) < R \} \subset \Omega.
\end{equation}
Observe that
\begin{equation}
\label{interior and boundary}
\partial \{ y \in \mathbb R^N : \mbox{ dist}(y, \overline{D}) < R \} = \{ y \in \mathbb R^N : \mbox{ dist}(y, \overline{D}) = R \}.
\end{equation}
Since $\Omega$ is a domain, in view of \eqref{distance to D}, \eqref{distance to D and interior} and \eqref{interior and boundary}, we conclude that
\begin{equation}
\label{determination of the domain}
\{ y \in \mathbb R^N : \mbox{ dist}(y, \overline{D}) = R \} = \partial\Omega\ \mbox{ and } \{ y \in \mathbb R^N : \mbox{ dist}(y, \overline{D}) < R \} = \Omega,
\end{equation}
which yields the converse inclusion of \eqref{the left includes the right}. 
\end{proof}

\begin{lemma}
\label{le:reflection of D implies that of omega}
Let $\ell$ be a unit vector in $\mathbb R^N,\ \lambda \in \mathbb R,$ and let $\pi_\lambda$ be the hyperplane $x\cdot\ell = \lambda$. 
\par
Set $D_\lambda = \{ x \in D : x\cdot\ell > \lambda \}$ and $\Omega_\lambda = \{ x \in \Omega : x\cdot\ell > \lambda \},$
and denote by $D_\lambda^\prime$  and $\Omega_\lambda^\prime$ the reflection 
of $D_\lambda$ and $\Omega_\lambda$ in the plane $\pi_\lambda$,
respectively. 
\par
Under the assumptions of \ {\rm (I)} of Theorem \ref{th:nonlinear Matzoh revisited}, if $D_\lambda^\prime \subset D$, then $\Omega_\lambda^\prime \subset \Omega.$ 
\end{lemma}
\begin{proof}
Since $D_\lambda^\prime \subset D,$ then also the set $D_{\rm sym}=D_\lambda\cup (\pi_\lambda\cap D)\cup D_\lambda^\prime$
is contained in $D.$ Thus, by Lemma \ref{le:exterior sum}, there holds that
$$
D_{\rm sym}\cup\bigcup\limits_{x\in\partial D_{\rm sym}} B_R(x)\subset D\cup\bigcup\limits_{x\in\partial D} B_R(x) = \Omega,
$$
and hence $\Omega_\lambda^\prime \subset \Omega.$
\end{proof}
\vskip.1cm
We can now complete the proof of (I). Lemma \ref{le:reflection of D implies that of omega} allows us to apply the method of moving planes, instead of to either $\Omega$ or $\mathbb R^N \setminus \overline{\Omega}$ 
as in the proof of \cite[Theorem 1.2]{MSpoincare}, directly to either $D$ or $\mathbb R^N \setminus \overline{D}.$  
Apart from this difference, the proof runs with the same arguments used in \cite[Theorem 1.2]{MSpoincare}. It is 
worth noticing that, by \cite[Section 5.2]{Fr}, the method of moving planes is applicable to $C^1$ domains, 
as $D$ is assumed to be.
\par
The proof of (II) is similar to that of  \cite[Theorem 1.3]{MSpoincare}.


\setcounter{equation}{0}
\setcounter{theorem}{0}

\section{Proof of Theorem \ref{th:Matzoh revisited}}
\label{section3}

We recall that $D$ satisfies the {\it interior cone condition}  with respect to $\Gamma,$ 
if for every $x \in \Gamma,$ there exists a finite  
right spherical open cone $K_x$ with vertex $x$ 
such that $K_x \subset D$ and $\overline{K_x}\cap \partial D =\{ x\}$.
\par
In view of the proof of \cite[Theorem 2.1]{MSpoincare}, Theorem \ref{th:Matzoh revisited} 
directly follows from the following lemma ---
note that this holds  for general domains $\Omega,$ including 
the case in which their boundaries  are unbounded. 
 
\begin{lemma}
\label{le:preliminary} Under the assumptions of Theorem \ref{th:Matzoh revisited}, 
the following assertions hold, even if $\partial\Omega$ is unbounded.
\begin{enumerate}[\rm (1)]
\item There exists a number $R > 0$ such that $d(x) = R$ for every $x \in \Gamma$; 
\item $\Gamma$ is a real analytic hypersurface;
\item there exists a connected component $\gamma$ of $\pa\Om,$ that is also a real analytic hypersurface,
such that the mapping $\gamma \ni \xi \mapsto x(\xi) \equiv \xi + R\nu(\xi) \in \Gamma,$ 
where $\nu(\xi)$ is the inward unit normal vector to $\pa\Om$ at $\xi \in \gamma$, is a diffeomorphism; 
in particular, $\gamma$ and $\Gamma$ are parallel hypersurfaces at distance $R$;
\item it holds that
\begin{equation}
\label{bounds of curvatures}
 \max_{1\le j \le N-1}\kappa_j(\xi) < \frac 1R\ \mbox{ for every } \xi \in \gamma,
\end{equation}
where $\kappa_1(\xi), \cdots, \kappa_{N-1}(\xi)$ are the principal curvatures of $\pa\Om$ at $\xi \in \gamma$ with respect to the inward unit normal vector to $\pa\Om$;
\item there exists a number $c > 0$ such that
\begin{equation}
\label{monge-ampere}
\prod_{j=1}^{N-1} \left(\frac 1R-\kappa_j(\xi)\right) = c\quad\mbox{ for every } \xi \in \gamma.
\end{equation}
\end{enumerate}
\end{lemma}

\noindent
{\bf Remark. } We emphasize a new important fact in this lemma: here we {\it do not} assume 
the exterior sphere condition on $\Omega,$ that was needed in previous papers (\cite{MSannals, MSindiana, MSpoincare}) 
to show that $\gamma$ and $\Gamma$ are parallel hypersurfaces. 
Thus, for example, with the aid of this lemma, we can remove the exterior-sphere-condition
assumption from all the theorems 
\cite[Theorem 1.1]{MSannals}, \cite[Theorem 3.1]{MSindiana} and \cite[Theorem 2.1]{MSpoincare},
and obtain very general characterizations 
of the sphere in terms of stationary isothermic surfaces.
\par
Therefore, the occurrence of stationary isothermic surface is a very strong requirement indeed.


\vskip 2ex
\noindent
{\it Proof of Lemma \ref{le:preliminary}. }  First of all, (1) follows from  \eqref{stationary level 3} and Theorem \ref{th:nonlinear-Varadhan}. 
 (2) follows from almost the same argument as in  (ii) of Lemma 3.1 of \cite{MSannals}. Since here we also deal with the Cauchy problem and $\partial \Omega$ is not necessarily bounded, for the reader's convenience we give a proof by using Theorem \ref{th:nonlinear-Varadhan} directly, instead of dealing with the Laplace transform of the solution as in \cite{MSannals}.  Besides Theorem \ref{th:nonlinear-Varadhan}, we use the balance law with respect to stationary critical points  of the solution, the interior cone condition of $D$ together with \eqref{nearest component}, \eqref{stationary level 3},  and (1). 
  
 It suffices to show that, for every point $x\in \Gamma,$
there exists  a time $t > 0$ such that
$\nabla u(x,t)\not= 0;$
then, (2) follows from \eqref{stationary level 3}, analyticity
of $u$ with respect to the space variable, and the implicit function theorem. We  use a balance law with respect to stationary critical points  of the caloric functions stated as follows
(see \cite{MSannals} for a proof): Let $G$ be a domain in $\mathbb R^N$. For $x_0\in G$,  a solution $v = v(x,t)$ of the heat equation in $G \times (0,+\infty)$ is such that $\nabla v(x_0,t)=0\ $ for every $t > 0$ if and only if
\begin{equation}
\label{eq:firstbalance}
\int\limits_{\pa B_r(x_0)}(x-x_0)v(x,t)\ dS_x=0\
\mbox{for every } (r,t) \in [0,\mbox{ dist}(x_0,\pa G)) \times (0,+\infty).
\end{equation}

Assume by contradiction that there exists a point $x_0\in\Gamma$
such that
$\nabla u(x_0,t)= 0$ for every $t > 0.$
By \eqref{eq:firstbalance} we can infer that
\begin{equation}
\label{balancelaw of critical point}
\int\limits_{\partial B_r(x_0)} (x-x_0)\ u(x,t)\ dS_x = 0\ \mbox{ for every } (r,t) \in (0, R) \times (0,+\infty).
\end{equation}
 Here let us choose $r = \frac R2$.

On the other hand, since $D$ satisfies the interior cone condition, there
exists a finite right spherical open 
cone $K$ with vertex at $x_0$ such that $K \subset D$
and $\overline{K} \cap \partial D
= \{x_0\}.$
By translating and rotating if needed, we can suppose that $x_0 = 0$
and that $K$ is the set $\{ x \in B_\rho(0) : x_N < -|x|\cos\theta \},$
where $\rho
\in (0,\frac R2)$ and  $\theta \in (0,\frac \pi 2).$

Since $K \subset D$ and $\overline{K} \cap
\partial D = \{0\},$  (1) and \eqref{nearest component} imply that
\begin{equation}
\label{inthecone}
d(x) > R\ \mbox{ for every } x \in K.
\end{equation}
Let us set
\begin{equation}
\label{dualcone}
 V_s = \{ x \in \partial B_{s}(0) : x_N \ge s \sin\theta \}\ \mbox{ for } s > 0.
\end{equation}
Then 
\begin{equation}
\label{contactset}
\partial\Omega \cap \partial B_R(0) \subset V_R,
\end{equation}
because, otherwise, there would be a point in $K$ contradicting
(\ref{inthecone}).

Thus, from (\ref{contactset}) it follows that we can choose a small number
$\delta >0$ such that
\begin{equation}
\label{away}
d(x) \ge \frac R2 + 2\delta\ \mbox{ for every } x \in \partial B_{\frac R2}(0) \cap \{ x_N \le 0 \}.
\end{equation}
Since, by Theorem \ref{th:nonlinear-Varadhan}, 
$-4t\log u(x,t)$ converges uniformly on
$\partial B_{\frac R2}(0)$ to $d(x)^2$ as
$t \to 0^+,$ we can choose
$t^* > 0$ such that
$$
\left|-4t\log u(x,t) -d(x)^2\right|<\delta^2\ \mbox{ for every } (x,t) \in \partial B_{\frac R2}(0)\times (0,t^*).
$$
This latter inequality, together with (\ref{dualcone}),
(\ref{contactset}), and (\ref{away}),  gives, for every $t \in (0,t^*),$
the following two estimates:
\begin{eqnarray}
&&\int\limits_{\partial B_{\frac R2}(0) \cap \{ x_N \le 0 \}}
x_N\ u(x,t)\ dS_x\ge
-\frac R4\ e^{-{\frac 1{4t}}(\frac {R^2}4+2R\delta+ 3\delta^2)}\ \mathcal H^{N-1}\left(\partial B_{\frac R2}(0)\right),\label{estimate-}
 \\
&&\int\limits_{V_{\frac R2} \cap \overline{\Omega_{\frac R2 +\delta}}} x_N\ u(x,t)\ dS_x\ge
\frac R2\sin\theta\  e^{-{\frac 1{4t}}(\frac {R^2}4+R\delta+ 2\delta^2)}\mathcal H^{N-1}\left(V_{\frac R2} \cap \overline{\Omega_{\frac R2 +\delta}}\right).\label{estimate+}
\end{eqnarray}
Here $\mathcal H^{N-1}(\cdot)$ denotes the $(N-1)-$dimensional Hausdorff measure and
$\Omega_{\frac R2 +\delta}$ is defined by
$$
\Omega_{\frac R2 +\delta} = \left\{ x \in \Omega : d(x) < \frac R2 +\delta \right\}.
$$
A consequence of \eqref{estimate-} and \eqref{estimate+}  is that,
for every  $t \in (0,t^*),$
\begin{eqnarray*}
&&\int\limits_{\partial B_{\frac R2}(0)} x_N\ u(x,t)\ dS_x 
\\
&&\ge \int\limits_{V_{\frac R2} \cap \overline{\Omega_{\frac R2 +\delta}}} x_N\ u(x,t)\ dS_x
+ \int\limits_{\partial B_{\frac R2}(0) \cap
\{x_N \le 0\}} x_N\
u(x,t)\ dS_x 
\\
&& \ge \frac R4  e^{-{\frac 1{4t}}(\frac {R^2}4+R\delta+ 2\delta^2)}\left[ 2\sin\theta \mathcal H^{N-1}\left(V_{\frac R2} \cap \overline{\Omega_{\frac R2 +\delta}}\right) -
e^{-{\frac 1{4t}}(R\delta+ \delta^2)}\ \mathcal H^{N-1}\left(\pa B_{\frac R2}(0)\right)\right].
\end{eqnarray*}
Therefore, we obtain a contradiction by observing that the first term of this
chain of inequalities equals zero, by \eqref{balancelaw of critical point} with $r = \frac R2$,
while the last term can be made positive
by choosing $t > 0$ sufficiently small. This completes the proof of (2).

Now, based on (1) and (2), let us  prove (3), (4) and (5) at the same time without assuming the exterior sphere condition on $\Omega,$ as pointed out in the previous remark.
\par
In view of (2), let $\nu_\Gamma(x)$ and $\hat\kappa_j(x), \cdots, \hat\kappa_{N-1}(x)$ be the unit outward normal vector to $\partial D$ at $x \in\Gamma$ and the principal curvatures of $\Gamma$ at $x \in\Gamma$ with respect to $\nu_\Gamma(x)$, respectively. Notice that, in view of \eqref{nearest component},  (2) and (1) imply that
\begin{equation}
\label{one-to-one}
\mbox{ for each $x \in \Gamma\ $ there exists a unique $\xi \in \partial\Omega$ satisfying $x \in \partial B_R(\xi)$}.
\end{equation}
Moreover, $\xi = x + R\nu_\Gamma(x)$, and in view of (1) and (\ref{one-to-one}), comparing  the principal curvatures at $x$ of $\Gamma$ with those of the sphere $\partial B_R(\xi)$ yields that
\begin{equation}
\label{curvature-weakinequalityGamma}
\max_{1\le j \le N-1} \hat\kappa_j(x) \le \frac 1R\quad\mbox{ for every }\  x \in \Gamma.
\end{equation}
\par
Since $\Gamma$ is a connected component of $\partial D$, $\Gamma$ is oriented and $\Gamma$ divides $\mathbb R^N$ into two domains. Let $E$ be the one of them which does not intersect $D$.  By  (1) and \eqref{nearest component},  $E \cap \left(\mathbb R^N \setminus \overline{\Omega}\right)$ contains a point, say,  $z.$ Set  $R_0= \mbox{ dist}(z, \Gamma) $. Then $R_0 > R$ and there exists a point $p_0 \in \Gamma$ such that $R_0 = |z-p_0|.$ Comparing the principal curvatures at $p_0$ of $\Gamma$ with those of the sphere $\partial B_{R_0}(z),$ yields that $\displaystyle \hat\kappa_j(p_0) \le \frac 1{R_0} < \frac 1R$ for every $j = 1, \dots, N-1$. By continuity, there exists $\delta_0 > 0$ such that  
\begin{equation}
\label{curvature-strictinequalityGamma}
\max_{1\le j \le N-1} \hat\kappa_j(x) < \frac 1R\quad\mbox{ for every }\  x \in \Gamma\cap \overline{B_{\delta_0}(p_0)}.
\end{equation}
This fact guarantees that the mapping: $B_{\delta_0}(p_0)\cap \Gamma \ni x \mapsto \xi(x) \equiv x + R\nu_\Gamma(x) \in \partial\Omega$ is a local diffeomorphism, that is, by letting 
$P_0 = p_0+R\nu_\Gamma(p_0)$, we can find a neighborhood $U$ of $P_0$ in $\mathbb R^N$ such that the mapping: $B_{\delta_0}(p_0)\cap \Gamma\ni x \mapsto \xi(x) \in U \cap \partial\Omega$ is a diffeomorphism. Moreover, since $\Gamma$ is a real analytic hypersurface because of (2), this diffeomorphism is also real analytic. Hence, $U \cap \partial\Omega$ is a portion of a real analytic hypersurface. 
\par
Notice that
the principal curvatures $\kappa_1(\xi), \cdots, \kappa_{N-1}(\xi)$ of $\partial\Omega$ at $\xi \in U \cap \partial\Omega$ with respect to the inward unit normal vector to $\partial\Omega$ satisfy
$$
-\kappa_j(\xi(x)) = \frac {\hat\kappa_j(x)}{1-R\hat\kappa_j(x)} \ \mbox{ for every } j = 1, \dots, N-1 \mbox{ and every }\ x \in \Gamma\cap B_{\delta_0}(p_0).
$$
Therefore, since $1-R\kappa_j(\xi(x)) = 1/(1-R\hat\kappa_j(x))$, we see that (\ref{curvature-strictinequalityGamma}) is equivalent to
\begin{equation}
\label{curvature-strictinequalityOmega}
\max_{1\le j\le N-1} \kappa_j(\xi) < \frac 1R\quad\mbox{ for every }\  \xi \in \overline{U} \cap \partial\Omega.
\end{equation}

We now use another balance law with respect to stationary zeros of the caloric functions stated as follows
(see \cite{MSannals} for a proof): Let $G$ be a domain in $\mathbb R^N$. For $x_0\in G$,  a solution $v = v(x,t)$ of the heat equation in $G \times (0,+\infty)$ is such that $v(x_0,t)=0\ $ for every $t > 0$ if and only if
\begin{equation}
\label{eq:secondbalance}
\int\limits_{\pa B_r(x_0)}v(x,t)\ dS_x=0\
\mbox{for every } (r,t) \in [0,\mbox{ dist}(x_0,\pa G)) \times (0,+\infty).
\end{equation}

Let $P, Q \in U \cap \partial\Omega$ be two distinct points, and let $p, q \in B_{\delta_0}(p_0)\cap\Gamma$ be the points such that
$\xi(p) = P$ and $\xi(q) = Q$. Then, by (\ref{one-to-one}) we have
$$
\overline{B_R(p)}\cap\partial\Omega = \{P\}\ \mbox{ and }\ \overline{B_R(q)}\cap\partial\Omega = \{Q\}.
$$
Consider the function $v = v(x,t)$ defined by
$$
v(x,t) = u(x+p,t) - u(x+q,t)\ \mbox{ for }\ (x,t) \in B_R(0) \times (0,+\infty).
$$
Since $v$ satisfies the heat equation and, by \eqref{stationary level 3}, $v(0,t) = 0$ for every $t > 0$, it follows from (\ref{eq:secondbalance}) that 
$$
\int\limits_{B_R(p)} u(x,t)\ dx =  \int\limits_{B_R(q)} u(x,t)\ dx\ \mbox{ for every }\ t > 0.
$$
Therefore, by Theorem \ref{th:interaction curvatures} and \eqref{curvature-strictinequalityOmega},  multiplying both sides by $t^{-\frac{N+1}4 }$ and letting $t \to 0^+$ yield that
$$
\prod_{j=1}^{N-1}\left(\frac 1R-\kappa_j( P)\right) = \prod_{j=1}^{N-1}\left(\frac 1R-\kappa_j(Q)\right).
$$
Hence, it follows that there exists a constant $c > 0$ such that
\begin{equation}
\label{monge-ampere-formula-local}
\prod_{j=1}^{N-1}\left(\frac 1R-\kappa_j(\xi)\right) = c\ \mbox{ for every }\ \xi \in U \cap \partial\Omega.
\end{equation}
Since $1-R\kappa_j(\xi(x)) = 1/(1-R\hat\kappa_j(x))$, we see that
$$
\prod_{j=1}^{N-1}\left(\frac 1R-\hat\kappa_j(x)\right) = \frac 1{cR^{2(N-1)}}\ ( >0)\ \mbox{ for every }\ x \in B_{\delta_0}(p_0) \cap \Gamma.
$$
Moreover, analyticity of $\Gamma$ yields that this equality holds also for every $x \in \Gamma$ and hence by (\ref{curvature-weakinequalityGamma})
 \begin{equation}
\label{curvature-strictinequality-wholeGamma}
\max_{1\le j \le N-1} \hat\kappa_j(x) < \frac 1R\quad\mbox{ for every }\  x \in \Gamma.
\end{equation}
\par
Therefore, with the aid of this strict inequality, by setting
\begin{equation}
\label{definition of gamma}
\gamma = \{ \xi(x) \in \mathbb R^N : x \in \Gamma \},
\end{equation}
we see that the mapping:  $\Gamma \ni x \mapsto \xi(x) \equiv x + R\nu_\Gamma(x) \in \gamma$ is a real analytic diffeomorphism because of analyticity of $\Gamma$ and $\gamma$ is a connected component of
$\partial\Omega$ which is a real analytic hypersurface. Since the mapping: $\gamma \ni \xi \mapsto x(\xi) \equiv \xi + R\nu(\xi) \in \Gamma$ is the inverse mapping of the previous diffeomorphism, (3) holds.  (4) follows from (\ref{curvature-strictinequality-wholeGamma}). 
\par
Finally, combining  analyticity of $\gamma$ with (\ref{monge-ampere-formula-local}) yields (\ref{monge-ampere}).
The proof is complete. \qed


\setcounter{equation}{0}
\setcounter{theorem}{0}

\def\theequation{B.\arabic{equation}}
\def\thetheorem{B.\arabic{theorem}}

\appendix
\section*{Appendix}
\label{appendix}
Here, for the reader's convenience, we give a proof of Theorem \ref{th:interaction curvatures} for the heat equation provided
$\displaystyle{ \max_{1\le j \le N-1} \kappa_j(y_0) < \frac 1R}$ by using some idea and a geometric lemma of \cite{MSprsea}.

\vskip 2ex
\noindent
{\bf Proof of  Theorem \ref{th:interaction curvatures} for the heat equation provided
$\displaystyle{ \max_{1\le j \le N-1} \kappa_j(y_0) < \frac 1R}$.}

Set $\phi(s) \equiv s$. 
We distinguish two cases: 
$$
{\rm (I)}\  \partial\Omega\ \mbox{ is bounded and of class $C^2$;}\quad {\rm (II)}\  \partial\Omega\ \mbox{ is otherwise.}
$$
\par
Let us first show how we obtain case (II) once we have proved case (I). 
Indeed, we can find two $C^2$ domains, say $\Omega_1$ and $\Omega_2,$ with bounded boundaries, 
and a ball $B_\delta(y_0)$ with the following properties:
$\Omega_1$ and $\mathbb R^N \setminus \overline{\Omega_2}$ are bounded; 
$B_R(x_0)\subset\Omega_1 \subset \Omega \subset \Omega_2;$
\begin{equation*}
\label{two constructed domains}
B_\delta(y_0) \cap \partial\Omega \subset \partial\Omega_1\cap\partial\Omega_2\ 
\mbox{ and }\ \overline{B_R(x_0)} \cap \left(\mathbb R^N\setminus\Omega_i \right)= \{y_0\}\ \mbox{ for } i=1, 2.
 \end{equation*}
 \par
Let $u_i=u_i(x,t)\ (i=1,2)$ be the two bounded solutions of either problem \eqref{diffusion}-\eqref{initial} or problem \eqref{cauchy} where $\Omega$ is replaced by $\Omega_1$ or $\Omega_2$, respectively.
Since $\Omega_1 \subset \Omega \subset \Omega_2$, it follows from the comparison principle that
$$
u_2 \le u \ \mbox{ in }\ \Omega \times (0,+\infty) \ \mbox{ and }\ u \le u_1\ \mbox{ in }\ \Omega_1\times(0,+\infty).
$$
Therefore, it follows that for every $t > 0$
$$
t^{-\frac{N+1}4 }\!\!\!\int\limits_{B_R(x_0)}\! u_2(x,t)\ dx \le t^{-\frac{N+1}4 }\!\!\!\int\limits_{B_R(x_0)}\! u(x,t)\ dx \le t^{-\frac{N+1}4 }\!\!\!\int\limits_{B_R(x_0)}\! u_1(x,t)\ dx.
$$
These two inequalities show that case (I) implies case (II).

Hereafter, we assume that $\partial\Omega$ is bounded and of class $C^2$.  Let us consider the signed distance function $d^*= d^*(x)$ 
of $x \in \mathbb R^N$ to the boundary $\partial\Omega$ defined by
\begin{equation}
\label{signed distance}
d^*(x) = \left\{\begin{array}{rll}
 \mbox{ dist}(x,\partial\Omega)\ &\mbox{ if }\ x \in \Omega,
\\
-\mbox{ dist}(x,\partial\Omega)\ &\mbox{ if  }\ x \not\in \Omega.
\end{array}\right.
\end{equation}
Since $\partial\Omega$  is bounded and of class $C^2$,  there exists a number $\rho_0 > 0$ such that $d^*(x)$ is $C^2$-smooth on a compact neighborhood  $\mathcal N$ of the boundary $\partial\Omega$ given by
\begin{equation}
\label{neighborhood of boundary from both sides}
\mathcal N = \{ x \in \mathbb R^N : -\rho_0 \le d^*(x) \le \rho_0 \}.
\end{equation}
We write for $s > 0$
$$
\Omega_s = \{ x \in \Omega : d^*(x) < s \}\  \left( = \{ x \in \mathbb R^N : 0 < d^*(x) < s \} \right).
$$
Introduce a function $F = F(\xi)$ for $\xi \in \mathbb R$ by
$$
F(\xi) = \frac 1{2\sqrt{\pi}} \int_\xi^\infty e^{-s^2/4} ds.
$$
Then $F$ satisfies
\begin{eqnarray*}
&&F^{\prime\prime} + \frac 12\xi F^\prime = 0\  \mbox{ and }  F^\prime < 0\ \mbox{ in } \mathbb R,
\\
&& F(-\infty) = 1,\ F(0) = \frac 12, \ \mbox{ and } F(+\infty) = 0.
\end{eqnarray*}
For each $\varepsilon \in (0,1/4)$, we define two functions $F_\pm = F_\pm(\xi)$ for $\xi \in \mathbb R$ by
$$
F_\pm(\xi) = F(\xi \mp 2\varepsilon).
$$
Then $F_\pm$ satisfies
\begin{eqnarray*}
&&F_\pm^{\prime\prime} + \frac 12\xi F_\pm^\prime = \pm\varepsilon F_\pm^\prime,\  F_\pm^\prime < 0\  \mbox{ and } F_- < F < F_+\ \mbox{ in } \mathbb R,
\\
&& F_\pm(-\infty) = 1,\ F_\pm(0) \gtrless   \frac 12, \ \mbox{ and } F_\pm(+\infty) = 0.
\end{eqnarray*}

By setting
\begin{equation}
\label{def-of-pre-subsupersolutions}
v_{\pm}(x,t) = F_{\pm}\left(t^{-\frac 12}d^*(x)\right)\ \mbox{ for }\ (x,t) \in \mathbb R^N \times (0,+\infty),
\end{equation}
we obtain 
\begin{lemma}
\label{lm:pre-subsupersolutions}
For each $\varepsilon \in (0,1/4),$ there exists $t_{1,\varepsilon} >0$ satisfying
 $$
 (\pm1)\left\{(v_\pm)_t - \Delta v_\pm\right\} > 0\quad\mbox{ in }\ \mathcal N \times (0,t_{1,\varepsilon}].
 $$
\end{lemma}
\noindent
{\it Proof.} A straightforward computation gives   
$$
 (v_\pm)_t - \Delta v_\pm = -\frac 1t \left(\pm\varepsilon +\sqrt{t}\Delta d^*\right)\ F_\pm^\prime \quad\mbox{ in }\ \mathcal N \times (0,+\infty).
 $$
Then, for each $\varepsilon \in (0,1/4)$, by setting
$
t_{1,\varepsilon} = \left(\frac \varepsilon{2M}\right)^2,
$
where $M = \max\limits_{x \in \mathcal N} |\Delta d^*(x)|$, we complete the proof. \qed 
\par
Set $\rho_1= \max\{ 2R, \rho_0\}$. Let $u$ be the solution of problem \eqref{diffusion}-\eqref{initial} or problem \eqref{cauchy}. By Theorem \ref{th:nonlinear-Varadhan}, we  have that 
\begin{equation}
\label{Varadhan}
-4t\log u(x,t) \to d^*(x)^2\ \mbox{ as }\ t \to 0^+\ \mbox{ uniformly on }\ \overline{\Omega_{\rho_1 }\setminus\mathcal N}.
\end{equation}
Then, in view of this and the definition (\ref{def-of-pre-subsupersolutions}) of $v_\pm,$ we have
\begin{lemma}
\label{lm:estimate-inside}
Let $u$ be the solution of problem \eqref{diffusion}-\eqref{initial} or problem \eqref{cauchy}.
There exist three positive constants $t_0,\ E_1$ and $E_2$ satisfying
$$
\max\{ |v_+|,\ |v_-|,\ |u| \} \le E_1 e^{-\frac {E_2}t}\ \mbox{ in }\ \overline{\Omega_{\rho_1 }\setminus\mathcal N} \times (0,t_0].
$$
\end{lemma}
\noindent
{\it Proof.} If we choose $t_0 \in (0,\left(\frac{\rho_0}4\right)^2]$, then by (\ref{def-of-pre-subsupersolutions}) 
we can show the desired inequalities for $v_\pm$. As for $u$, by Theorem \ref{th:nonlinear-Varadhan}, we can take $t_0 > 0$ such that
$$
\left|4t\log u(x,t) + d^*(x)^2\right| < \frac 12 \rho_0^2\ \mbox{ for }\ (x,t) \in  \overline{\Omega_{\rho_1 }\setminus\mathcal N}\times (0,t_0],
$$
and hence
$$
u(x,t) < e^{-\frac {d^*(x)^2-\frac 12 \rho_0^2}{4t}}\mbox{ for }\ (x,t) \in\overline{\Omega} \times (0,t_0].
$$
Since $d^*(x) \ge \rho_0$ for $x \in  \overline{\Omega_{\rho_1 }\setminus\mathcal N}$, we get the desired inequality for $u$. \qed 
\par
By setting, for $(x,t) \in \mathbb R^N \times (0,+\infty)$,
\begin{equation}
\label{subsupersolutions-for-heat}
w_{\pm}(x,t) = \left\{\begin{array}{rl}  2v_{\pm}(x,t)\pm 2 E_1 e^{-\frac {E_2}t} &\mbox{ for problem \eqref{diffusion}-\eqref{initial} }, \\
(1\pm\varepsilon) v_{\pm}(x,t)\pm 2E_1 e^{-\frac {E_2}t} &\mbox{ for problem \eqref{cauchy}},
\end{array}\right. 
\end{equation}
we have 
\begin{lemma}
\label{lm:subsupersolutions-for-heat}
Let $u$ be the solution of problem \eqref{diffusion}-\eqref{initial} or problem \eqref{cauchy}.
For each $\varepsilon \in (0,1/4),$ there exists $t_\varepsilon > 0$ satisfying
\begin{equation}
\label{estimates from both}
w_- \le u \le w_+\quad\mbox{ in }\ \overline{\Omega_{\rho_1}} \times(0,t_\varepsilon].
\end{equation}
where $w_\pm$ are defined by {\rm (\ref{subsupersolutions-for-heat})}.
\end{lemma}
\noindent
{\it Proof.} For each $\varepsilon \in (0,1/4),$ we set 
$$
t_{2,\varepsilon} = \min\{ t_{1,\varepsilon}, t_0\}.
$$
Since $v_+,\ v_-,\ $ and $u$ are all nonnegative, Lemma \ref{lm:estimate-inside} implies that 
\begin{equation}
\label{inside-inequality}
w_- \le u \le w_+\quad\mbox{ in }\ \overline{\Omega_{\rho_1}\setminus\mathcal N} \times(0,t_{2,\varepsilon}].
\end{equation}
Let $u$ the solution of problem \eqref{diffusion}-\eqref{initial}.
Observe that 
\begin{eqnarray}
&&w_- \le u \le w_+\quad\qquad\ \ \mbox{ on }\ \partial\Omega\times (0,t_{2,\varepsilon}],\label{on-boundary}
\\
&&w_- = u = w_+ = 0\qquad\mbox{ on }\ \left(\Omega\cap\mathcal N\right) \times\{0\}.\label{at-initial}
\end{eqnarray}
Therefore, with the aid of the comparison principle and in view of Lemma \ref{lm:pre-subsupersolutions},
(\ref{inside-inequality}), (\ref{on-boundary}), and (\ref{at-initial}), we obtain \eqref{estimates from both} by setting $t_\varepsilon = t_{2,\varepsilon}$. 

It remains to consider the solution $u$ of problem \eqref{cauchy}. In view of the fact that $F_\pm(-\infty) = 1$, there exists $t_{3,\varepsilon} 
\in (0, t_{2,\varepsilon}]$ such that
\begin{equation}
\label{on-boundary-cauchy}
w_- < u < w_+\ \mbox{ on }\ \left(\partial\mathcal N \setminus \Omega\right) \times (0,t_{3,\varepsilon}].
\end{equation}
Observe also that
\begin{equation}
\label{at-initial-cauchy}
w_- \le u \le w_+ \qquad\mbox{ on }\ \mathcal N\times\{0\}.
\end{equation}
Therefore, with the aid of the comparison principle and in view of Lemma \ref{lm:pre-subsupersolutions},
(\ref{inside-inequality}), (\ref{on-boundary-cauchy}), and (\ref{at-initial-cauchy}), we obtain \eqref{estimates from both} by setting $t_\varepsilon = t_{3,\varepsilon}$. \qed

By writing
$$
\Gamma_s = \{ x \in \Omega : d^*(x) = s \}\ \mbox{ for } s > 0,
$$
let us quote a geometric lemma from \cite{MSprsea} adjusted to our situation.
\begin{lemma}{\rm (\cite[Lemma 2.1, p. 376]{MSprsea})}
\label{lm:asympvol} If $\displaystyle{ \max_{1\le j \le N-1} \kappa_j(y_0) < \frac 1R}$, then
we have:
\begin{equation*}
\label{asympvol}
\lim_{s\to 0^+} s^{-\frac{N-1}{2}} \mathcal H^{N-1}(\Gamma_s\cap B_R(x_0))=2^{\frac{N-1}{2}}\omega_{N-1} \
\left\{\prod_{j=1}^{N-1}\left(\frac1{R}-\kappa_j(y_0)\right)\right\}^{-\frac12},
\end{equation*}
where $\mathcal H^{N-1}$ is the standard $(N-1)$-dimensional Hausdorff measure, and $\omega_{N-1}$ is the volume of the unit ball in $\mathbb R^{N-1}.$
\end{lemma}
Since $\displaystyle{ \max_{1\le j \le N-1} \kappa_j(y_0) < \frac 1R}$, we can use this lemma.
 
Lemma \ref{lm:subsupersolutions-for-heat} implies that for every $t \in (0,t_\varepsilon]$
\begin{equation}
\label{estimates from both sides}
t^{-\frac{N+1}4 }\int_{B_R(x_0)} w_- \ dx \le t^{-\frac{N+1}4 } \int_{B_R(x_0)} u \ dx \le t^{-\frac{N+1}4 }\int_{B_R(x_0)} w_+ \ dx.
\end{equation}
Also, with the aid of the co-area formula, we have:
\begin{equation}
\label{apply co-area formula}
\int_{B_R(x_0)} v_\pm\ dx = t^{\frac{N+1}4}\int_0^{2Rt^{-\frac 12}} F_\pm(\xi)\xi^{\frac {N-1}2} \left(t^{\frac 12}\xi\right)^{-\frac {N-1}2} \mathcal H^{N-1}\left(\Gamma_{t^{\frac12}\xi}\cap B_R(x_0)\right) d\xi,
\end{equation}
where $v_\pm$ is defined by \eqref{def-of-pre-subsupersolutions}.

First, we take care of problem  \eqref{diffusion}-\eqref{initial}.
By Lebesgue's dominated convergence theorem, Lemma \ref{lm:asympvol} and \eqref{apply co-area formula},  we get
\begin{equation}
\label{limit as t to 0}
\lim_{t\to 0^+} t^{-\frac{N+1}4}\int_{B_R(x_0)} w_\pm\ dx = 2^{\frac{N-1}2}\omega_{N-1}\left\{\prod_{j=1}^{N-1}\left(\frac 1R-\kappa_j(y_0)\right)\right\}^{-\frac 12}\int_0^\infty 2 F_\pm(\xi) \xi^{\frac{N-1}2} d\xi.
\end{equation}
Moreover, again by Lebesgue's dominated convergence theorem, we see  that
\begin{equation}
\label{limit as epsion to 0}
\lim_{\varepsilon \to 0}\int_0^\infty 2F_\pm(\xi) \xi^{\frac{N-1}2} d\xi = \int_0^\infty 2F(\xi) \xi^{\frac{N-1}2} d\xi.
\end{equation}
Therefore, since $\varepsilon > 0$ is arbitrarily small in \eqref{estimates from both sides},  it follows that \eqref{asymptotics and curvatures} holds true, where we set
\begin{equation}
\label{coefficient}
c(\phi,N) = 2^{\frac{N-1}2}\omega_{N-1}\int_0^\infty 2F(\xi) \xi^{\frac{N-1}2} d\xi.
\end{equation}

In the case of problem \eqref{cauchy}, the proof runs similarly by replacing $2F_\pm,\ 2F$ with $(1\pm \varepsilon) F_\pm,\ F$, respectively,  in \eqref{limit as t to 0}-\eqref{coefficient}.

\vskip 4ex
\bigskip
\noindent{\large\bf Acknowledgement.}
\smallskip
The main part of the present paper was written during a  stay of the second author at the Department
of Mathematics of the University of Florence. He gratefully acknowledges
its hospitality. The authors would also like to thank the referees for their many valuable suggestions to improve clarity in several points.

\end{document}